\documentclass[12pt,leqno]{article}
\usepackage[T1]{fontenc}
\usepackage{amsmath,amsthm}
\usepackage{amssymb,latexsym}
\usepackage{enumitem}
\newtheorem{thm}{Theorem}
\newtheorem{cor}{Corollary}
\newtheorem{lem}{Lemma}

\numberwithin{equation}{section}
\theoremstyle{definition}

\newtheorem{exa}{Example}
\setenumerate{label={\textup{(\alph*)}},nolistsep,fullwidth,itemindent=\parindent}
\usepackage[a4paper,body={5.75in,8.5in},nohead]{geometry}
\makeatletter
\def\@seccntformat#1{\csname the#1\endcsname.\ }
\makeatother
\begin{document}
\baselineskip=16pt
\author{Piotr  Nowak\\
Mathematical Institute,  University of Wroc{\l}aw\\
Pl. Grunwaldzki 2/4, 50-384 Wroc{\l}aw, Poland\\
E-mail: nowak@math.uni.wroc.pl
}
\title{Order preserving property of moment estimators}
\date{}
\maketitle
%%%%%%%%%%%%%%%%%%%%%%%%%%%%%%%%%%%%%%%%%%%%%%%%%%%%%%%%%
\begin{abstract}
Balakrishnan and Mi \cite{bal_01} considered order preserving
property of maximum likelihood estimators.
In this paper
there are given conditions under which
the moment estimators  have the property of preserving 
stochastic orders.
The  property of preserving for usual 
stochastic order as well as for likelihood ratio order is considered.
Sufficient conditions  are 
established for some parametric families of distributions.
\end{abstract}
%%%%%%%%%%%%%%%%%%%%%%%%%%%%%%%%%%%%%%%%%%%%%%%%%%%%%%%%%%%%%%%%%%%%%%%%
%%%%%%%%%%%%%%%%%%%%%%%%%%%%%%%%%%%%%%%%%%%%%%%%%%%%%%%%%%%%%%%%%%%%%%%%
%%%%%%%%%%%%%%%%%%%%%%%%%%%%%%%%%%%%%%%%%%%%%%%%%%%%%%%%%%%%%%%%%%%%%%%%
%%%%%%%%%%%%%%%%%%%%%%%%%%%%%%%%%%%%%%%%%%%%%%%%%%%%%%%%%%%%%%%%%
%%%%%%%%%%%%%%%%%%%%%%%%%%%%%%%%%%%%%%%%%%%%%%%%%%%%%%%%%%%%%%%%%
%%%%%%%%%%%%%%%%%%%%%%%%%%%%%%%%%%%%%%%%%%%%%%%%%%%%%%%%%%%%%%%%%
%%%%%%%%%%%%%%%%%%%%%%%%%%%%%%%%%%%%%%%%%%%%%%%%%%%%%%%%%%%%%%%%%
%%%%%%%%%%%%%%%%%%%%%%%%%%%%%%%%%%%%%%%%%%%%%%%%%%%%%%%%%%%%%%%%%%%%%%%%
%%%%%%%%%%%%%%%%%%%%%%%%%%%%%%%%%%%%%%%%%%%%%%%%%%%%%%%%%%%%%%%%%%%%%%%%
%%%%%%%%%%%%%%%%%%%%%%%%%%%%%%%%%%%%%%%%%%%%%%%%%%%%%%%%%%%%%%%%%%%%%%%%
\section{Introduction and preliminaries}

Suppose that $\mathbf{X}=(X_1,X_2,\ldots,X_n)$ is a sample from a 
population with density
$f(x;{\theta})$, where  $\theta  \in \Theta \subset \mathcal{R}$.
We now consider the estimation of $\theta$
by the method of moments (see, for example, Borovkov \cite{borovkov}).
Let $g:\mathcal{R}\to\mathcal{R}$  be a function such  that the
function $m:\Theta\to\mathcal{R}$
$$
m(\theta)=\int_{-\infty}^{\infty} g(x)f(x;\theta)dx
$$
is monotone  and continuous on $\Theta$.  Particularly, if $m$
is strictly monotone, then $m$ is one-to-one and  $m^{-1}$ exits.
Then for every $t\in m(\Theta)$  exists a unique solution 
of the equation $m(\theta)=t$, i.e.  $\theta=m^{-1}(t)$.
In the case when $m$ is nondecreasing or nonincreasing 
we can take $m^{-1}(t)=\inf\{\theta: m(\theta)\leq t\}$, for any $t\in m(\Theta)$.
Let 
$$
\bar g=\frac{1}{n}\sum_{i=1}^n g(X_i)
$$
be the generalized empirical moment based on $\mathbf{X}$.

If $\bar g\in m(\Theta)$, then the estimator obtained by the method 
of moments  (to be short \emph{moment estimator}) is of the form
\begin{equation}\label{mm:form}
\hat \theta = m^{-1} (\bar g).
\end{equation}

We know that
these estimators are strongly consistent.
We often put  $g(x)~=~x^k$, $k\geq 1$.
Then $\bar g$ reduces to $m_k$ --
$k$th empirical moment, i.e.~$m_k=\frac{1}{n}\sum_{i=1}^n X_i^k$.
Particularly, $m_1=\bar{X}$.

%%%%%%%%%%%%%%%%%%%%%%%%%%%%%%%%%%%%%%%%%%%%%%%%%%%%%%%%%
In this paper our aim is to give conditions under
which the moment estimators  have the property of preserving 
stochastic orders.  We shall deal with stochastic orders, so
recall their definitions.
%%%%%%%%%%%%%%%%%%%%%%%%%%%%%%%%%%%%%%%%%%%%%%%%%%%%%%%%%

Let $X$ and $Y$ be two random variables, $F$ and $G$ their
respective distribution functions and $f$ and $g$ their
respective density functions, if they exist. 
We say that $X$ is stochastically smaller
than $Y$ (denoted by $X\leq_{\textup{st}}Y$) if $F(x)\geq G(x)$ 
for all $x\in\mathcal{R}$. The stronger order than
usual stochastic order is likelihood ratio order.
 We say that $X$~is smaller than $Y$ in the likelihood
ratio order (denoted by $X\leq_{\textup{lr}}Y$)
if $g(x)/f(x)$ is increasing in $x$.
Let $F^{-1}$ and $G^{-1}$ be quantile functions of $F$
and $G$ respectively. We say that $X$ is less dispersed than
$Y$ (denoted $X\leq_{\textup{disp}}Y$) if $G^{-1}(\alpha)-F^{-1}(\alpha)$
is an increasing function in $\alpha\in(0,1)$.
The family of distributions $\{f(x;\theta),\theta\in\Theta\subset\mathcal{R}\}$
is stochastically  increasing in $\theta$ if
$X(\theta_1)\leq_{\textup{st}}Y(\theta_2)$
for all $\theta_1<\theta_2\in\Theta$, where 
$X(\theta)$ has a density $f(x;\theta)$. 
For example, it is well known that location
parameter family $\{f(x-\theta), \theta\in R\}$
and scale parameter family $\{1/\theta f(x/\theta),\theta>0\}$, 
$x>0$,  are stochastically increasing in $\theta$.
For further details on stochastic orders we refer 
to  Shaked and Shanthikumar \cite{shaked_st:or} and
 Marshall and Olkin \cite{olkin}.

Furthermore, we say that the family of distributions
$\{f(x;\theta),\theta\in\Theta\subset\mathcal{R}\}$ has 
monotone likelihood ratio if
$f(x;\theta_2)/f(x;\theta_1)$ is increasing function in $\theta$
for any $\theta_1<\theta_2\in\Theta$, i.e. 
$X(\theta_1)\leq_{\textup{lr}}X(\theta_2)$.
For example, the following families have the monotone likelihood ratio in $\theta$:
\begin{enumerate}
\item the family of distributions with density
 $f(x;\theta)=c(\theta)h(x)I_{(-\infty,\theta)}(x)$;
\item the family of distributions with density
 $f(x;\theta)=c(\theta)h(x)\exp(\eta(\theta)t(x))$;
provided that both $\eta$ and $t$ are increasing.
\end{enumerate}

The likelihood ratio order is  closely related with
total positivity, (see Karlin \cite{karlin_68}).
Let $k(x,y)$ be a measurable function defined on $X\times Y$,
where $X$ and $Y$ are subsets of $\mathcal{R}$.
We say that $k(x,y)$
is totally positive of order $r$ 
(to be short $k(x,y)$ is $TP_r$)
if for all
$x_1<\dots<x_m$, $x_{i}\in X$; 
and for all  $y_1<\dots<y_m$, $y_{i}\in Y$;
and all $1\leq m\leq r$,
we have
$$
\left|
\begin{array}{ccc}
k(x_1,y_1)&\dots&k(x_1,y_m)\\
\vdots&&\vdots\\
k(x_m,y_1)&\dots&k(x_m,y_m)
\end{array}
\right|
\geq 0.
$$
%We say that $k$ is $STP_k$ if the above inequality is strict.

It is clear, that the ordering
$X(\theta_1)\leq_{\textup{lr}}X(\theta_2)$
whenever $\theta_1<\theta_2\in\Theta$ is equivalent that
the density $f(x;\theta)$ is $TP_2$.
We refer to Karlin \cite{karlin_68} for proofs of basic facts 
in the theory of total positivity.

From the definition immediately follows that if $g$ and $h$
are nonnegative  functions and $k(x,y)$ is $TP_r$ then
$g(x)h(y)k(x,y)$ is also $TP_r$.
Similarly, if $g$ and $h$ are increasing functions and $k(x,y)$ is $TP_r$,
then $k(g(x),h(y))$ is again $TP_r$.

Twice differentiable positive function $f(x;\theta)$
is  $TP_2$ if and only if
$$
\dfrac{\partial ^2}{\partial x \partial \theta} \log f(x;\theta) \geq 0.
$$

Total positivity of many functions that arise in statistics
follows from the basic composition formula. 
If $g(x,z)$ is $TP_m$, $h(z,y)$ is $TP_n$ and if the convolution
$$
k(x,y)=\int_{-\infty}^{\infty} g(x,z) h(z,y)dz
$$
is finite, then $k(x,y)$ is
 $TP_{\min (m,n)}$.

A particular and important case is when $k(x,y)=f(y-x)$.
A nonnegative  function $f$ is said to be $PF_k$ 
(P\'olya frequency function of order $k$) if $f(y-x)$ is $TP_k$.

Recall that a real valued function $f$ is said to be
 logconcave on  interval $A$ if $f(x)\geq 0$ and $\log f$ is an
extended real valued concave function (we put $\log 0 =-\infty$).
It is well known (see, for example, Schoenberg \cite{schoenberg})
that the function $f$ is $PF_2$ if and only if
$f$ is nonnegative and logconcave on~$\mathcal{R}$.
Recall also a  very important property of $PF_2$ functions 
which we will use in the sequel that
if $g$ and $h$ are logconcave functions on $\mathcal{R}$,
such that the convolution 
$$h(x)=\int_{-\infty}^{\infty}g(x-z)h(z)dz$$
is defined for all $x\in\mathcal{R}$,
then the function $h$ is also logconcave on~$\mathcal{R}$.

Another very important property of $TP$ functions is the
variation diminishing property.

\begin{lem}\label{t2}
Let  $g$ be given by absolutely convergent integral
$$
g(x)=\int_{-\infty}^{\infty} k(x,y)f(y)dy
$$
where $k(x,y)$ is $TP_r$ and $f$ change sign at most $j\leq r-1$ times.
Then $g$ changes sign at most $j$ times. Moreover, if
$g$ changes  sign $j$ times, then $f$ and $g$ have the same arrangement of signs.
\end{lem}

The basic tool for proving stochastic ordering of estimators
are order preserving properties of underlying stochastic orders,
for the proof see Shaked and Shanthikumar \cite{shaked_st:or}.

\begin{lem}
Assume that $g$ is an increasing function.
\begin{enumerate}
\item If $X\leq_{\textup{st}}Y$, then $g(X)\leq_{\textup{st}}g(Y)$.
\item If $X\leq_{\textup{lr}}Y$, then $g(X)\leq_{\textup{lr}}g(Y)$.
\end{enumerate}
\end{lem}

\begin{lem}\label{l2}
Let $X_1,\dots,X_n$ and  $Y_1,\dots,Y_n$ be independent random variables.
\begin{enumerate}
\item If $X_i\leq_{\textup{st}}Y_i$, $i=1,\dots,n$,
 then $X_1+\dots +X_n\leq_{\textup{st}}Y_1+\dots + Y_n$.

\item If $X_i\leq_{\textup{lr}}Y_i$, $i=1,\dots,n$,
 then $X_1+\dots +X_n\leq_{\textup{lr}}Y_1+\dots + Y_n$,
provided these random variables have logconcave densities.
\end{enumerate}
\end{lem}

%%%%%%%%%%%%%%%%%%%%%%%%%%%%%%%%%%%%%%%%%%%%%%%%
Let $\hat\theta$ be a moment estimator of $\theta$
based on a sample from population with density 
$f(x;\theta)$, $\theta\in\Theta\subset\mathcal{R}$.
 Let now 
$\hat\theta_1$ and $\hat\theta_2$ be 
the moment estimators obtained   on the basis
of the sample $\mathbf{X}$ from population with density $f(x;\theta_1)$
and of the sample $\mathbf{Y}$ from population with density $f(x;\theta_2)$ respectively,
where  $\theta_1<\theta_2\in\Theta$.
We say, that $\hat\theta$ is 
stochastically increasing in $\theta$ if the family  of its distributions
 is stochastically increasing in $\theta$, i.e.
$\hat\theta_1\leq_{\textup{st}}\hat\theta_2$.
We also shall be interested whether the stronger property holds
$\hat\theta_1\leq_{\textup{lr}}\hat\theta_2$, i.e.  
$\hat\theta$ is increasing in $\theta$  with respect to  likelihood ratio order.

Note, that in general the moment estimator may not be
stochastically monotone as the following example indicates.
%%%%%%%%%%%%%%%%%%%%%%%%%%%%%%%%%%%%%%%%%%%%%%%%
\begin{exa}
Consider a sample $\mathbf{X}=(X_1,\dots,X_n)$ from the uniform distribution 
on the interval $(-\theta,\theta)$, $\theta>0$.
Then $E(X)=0$ and the first moment contains no information about $\theta$.
We use the second moment. Then we have $E(X^2)=\theta^2/3=m(\theta)$.
Thus the moment estimator of $\theta$ is of the form
 $\hat\theta=\sqrt{3/n\sum_{i=1}^n X_i^2}$.
Since the random variable $X$ is not stochastically increasing in $\theta$
and  $m(\theta)$ is increasing for $\theta>0$, the estimator
$\hat\theta$ is not  stochastically monotone.
\end{exa}

%%%%%%%%%%%%%%%%%%%%%%%%%%%%%%%%%%%%%%%%%%%%%%%%%%%%%%%%%%%%%%%%%%%%%%%%
%%%%%%%%%%%%%%%%%%%%%%%%%%%%%%%%%%%%%%%%%%%%%%%%%%%%%%%%%%%%%%%%%%%%%%%%
%%%%%%%%%%%%%%%%%%%%%%%%%%%%%%%%%%%%%%%%%%%%%%%%%%%%%%%%%%%%%%%%%%%%%%%%
%%%%%%%%%%%%%%%%%%%%%%%%%%%%%%%%%%%%%%%%%%%%%%%%%%%%%%%%%%%%%%%%%
%%%%%%%%%%%%%%%%%%%%%%%%%%%%%%%%%%%%%%%%%%%%%%%%%%%%%%%%%%%%%%%%%
%%%%%%%%%%%%%%%%%%%%%%%%%%%%%%%%%%%%%%%%%%%%%%%%%%%%%%%%%%%%%%%%%%%%%%%%
%%%%%%%%%%%%%%%%%%%%%%%%%%%%%%%%%%%%%%%%%%%%%%%%%%%%%%%%%%%%%%%%%%%%%%%%
%%%%%%%%%%%%%%%%%%%%%%%%%%%%%%%%%%%%%%%%%%%%%%%%%%%%%%%%%%%%%%%%%%%%%%%%
\section{Results}
 Assume for simplicity that
for a given function $g$ the generalized empirical moment belongs to 
$m(\Theta)$, the domain of the values $m(\theta)$, $\theta\in\Theta$, for every resulting set of observations.

The following theorem gives a sufficient conditions for likelihood ordering of 
moment estimators for one-parameter family in the case when the estimator 
is based on the sample mean.

\begin{thm}\label{t3}
Assume that the function $m(\theta)=\int_{-\infty}^{\infty} x f(x;\theta)dx<\infty$
   for all~$\theta$,
$f(x;\theta)$ is  $TP_2$
and  $f(x;\theta)$ is logconcave in x.
Then the moment estimator $\hat\theta$
is increasing in $\theta$ with respect to likelihood ratio order.
\end{thm}
\begin{proof}
Since $f(x;\theta)$ is $TP_2$ we have
ordering $X_i(\theta_1)\leq_{\textrm{lr}}X_{i}(\theta_2)$
for all $\theta_1<\theta_2\in\Theta$ and $i=1,\ldots,n$.
From the variation diminishing 
property (see Lemma~\ref{t2}) we deduce that function 
$$
m(\theta)-c = \int_{-\infty}^{\infty} (x-c)f(x;\theta)dx
$$
changes sign at most one for any $c$,
from $+$ to $-$ if  occurs. This implies that
$m(\theta)$  is  increasing. Hence
the  moment estimator is of the form 
$\hat\theta=m^{-1}(\bar X)$.
Since $f$ is logconcave,
we conclude from Lemma~\ref{l2}~(b) that
$\bar X$ is increasing in $\theta$ with respect to  
likelihood ratio order.
Theorem follows from the preserving property of monotone likelihood
order  under increasing operations.
\end{proof}

If we weakness the assumptions of Theorem \ref{t3} we can obtain the following
theorem.

\begin{thm}\label{t4}
Assume that the function
 $m(\theta)=\int_{-\infty}^{\infty} x f(x;\theta)dx<\infty$
  for all $\theta$ and
$f(x;\theta)$ is $TP_2$.
Then the  moment estimator of $\theta$
is stochastically increasing in $\theta$.
\end{thm}

In the general case we can formulate the following corollary. 
We omit the proof because it follows by the same method as in Theorem \ref{t3}.
\begin{cor}\label{cor1}
Assume that the function $g$ is increasing and the 
integral $m(\theta)=\int_{-\infty}^{\infty} g(x)f(x;\theta)dx<\infty$  for all $\theta$.
If $f(x;\theta)$ is $TP_2$, then the moment estimator $\hat\theta$
of the form (\ref{mm:form}) is stochastically increasing
in $\theta$. Moreover, if the random variable $g(X)$ has 
logconcave density, then the estimator $\hat\theta$
is increasing in $\theta$ with respect to  monotone likelihood ratio order.
\end{cor}
The next Corollary follows immediately from Corollary \ref{cor1}
and form Proposition~2 due to An \cite{an}.
\begin{cor}\label{cor2}
Let $X$ be a random variable  with a density $f(x;\theta)$
and $g$ be a~function such the integral
 $m(\theta)=\int_{-\infty}^{\infty} g(x)f(x;\theta)dx<\infty$ for all $\theta$. Assume that the following conditions are satisfied: \newpage
\begin{enumerate}
\item $g$  is strictly increasing, concave and differentiable;
\item $|\frac{\partial}{\partial x}v(x)|$ is logconcave, where $v=g^{-1}$ ;
\item $f$ is logconcave,  decreasing on the support of~$X$;
\item $f(x;\theta)$ is $TP_2$.
\end{enumerate}
Then the moment estimator $\hat\theta$ is increasing in $\theta$ 
with respect to   likelihood ratio order.
\end{cor}

In many situations  we deal with one-parameter
exponential family with
densities of the form 
\begin{equation}\label{def_exp}
f(x;\theta)=h(x)c(\theta)\exp(\eta(\theta)T(x)), \ \theta\in\Theta,\ x\in(a,b).
\end{equation}
Recall the well known formula for moments of $T(X)$, see Berger and Casella \cite{berger}:
\begin{equation}\label{m_exp}
E_{\theta}(T(X))=-\frac{[\log c(\theta)]'}{[\eta(\theta)]'},
\end{equation}
\begin{equation}\label{v_exp}
Var_{\theta}(T(X))=\frac{-[\log c(\theta)]'' +E_{\theta}[T(X)]\cdot[\eta(\theta)]'' }{([\eta(\theta)]')^2}.
\end{equation}

It is easy to prove that if 
both $T$ and $\eta$ are  increasing (decreasing), then
$f(x;\theta)$ is $TP_2$ and from the variation diminishing property it follows
that $m(\theta)=E_{\theta}(T(X))$ is increasing  (decreasing).
Combining those facts with order preserving
properties of the usual stochastic order we can formulate the following theorem.

\begin{thm}\label{t5} For the one-parameter exponential family with densities of the form (\ref{def_exp}), where both
$\eta$ and $T$ are  increasing (decreasing), the moment estimator
$\hat\theta=m^{-1}(1/n\sum_{i=1}^n T(X_i))$ is stochastically increasing in $\theta$.
\end{thm}

Let us make the following observations.
 Now we consider  
the maximum likelihood estimation of $\theta$ for the one-parameter exponential family
(\ref{def_exp}).
Let $\mathbf{x}=(x_1,\dots,x_n)$ be a value of a sample $\mathbf{X}$.
The maximum likelihood function is of the form
$$
L(\mathbf{x};\theta)=\prod_{i=1}^n h(x_i) [c(\theta)]^n \exp\left(\eta(\theta)\sum_{i=1}^n T(x_i)\right).
$$ 
So we have
$$
\frac{\partial \log L(\mathbf{x};\theta)}{\partial \theta}=
 n [\log c(\theta)]' + [\eta(\theta)]' \sum_{i=1}^n T(x_i).
$$
It is clear that
$
\dfrac{\partial \log L(\mathbf{x};\theta)}{\partial \theta}=0 
$
if and only if
\begin{equation}\label{eq_mle}
\frac{1}{n}\sum_{i=1}^{n} T(x_i)=-\frac{[\log c(\theta)]'}{[\eta(\theta)]'}.
\end{equation}
Let $\hat\theta$ be a solution of the equation $(\ref{eq_mle})$.
 Then $\hat\theta$
is the maximum likelihood estimator (MLE) since using (\ref{v_exp}) we have
$$
\frac{\partial^2 \log L(\mathbf{x};\theta)}{\partial \theta^2}\bigg|_{\theta=\hat\theta}=
-n ([\log c(\theta)]'')^2|_{\theta=\hat\theta} \cdot Var_{\hat\theta}(T(X))<0.
$$
On the other hand $\hat\theta$ is the moment estimator, since 
$E_{\theta}(T(X))=-\dfrac{[\log c(\theta)]'}{[\eta(\theta)]'}$.

Thus from Theorem \ref{t5} we have the following result.

\begin{thm}\label{t6} For the one-parameter exponential family
 with densities of the form (\ref{def_exp}), where both
$\eta$ and $T$ are  increasing (decreasing), the maximum likelihood estimator
 $\hat\theta$ is stochastically increasing in $\theta$.
\end{thm}

Particularly, Theorems 2 and 3 of Balakrishnan and Mi \cite{bal_01} give   conditions
under which  maximum likelihood estimators
for the one-parameter exponential family of the form (\ref{def_exp})
are  stochastically increasing.
In the  Theorem \ref{t6} we have proved it  but under weaker assumptions.

\begin{exa}
Let  $\mathbf{X}=(X_1,\dots,X_n)$ be a sample from the distribution with 
density
$$
f(x;\theta)=\sqrt{\frac{\theta}{\pi x^3}}\exp (-\theta/x), \ \ x>0, \ \theta>0.
$$
This is one-parameter exponential family with $T(x)=1/x$ and $\eta(\theta)=-\theta$. 
Using (\ref{m_exp}) we get moment estimator $\hat\theta=(2/n\sum_{i=1}^n 1/X_i)^{-1}$.
By Theorem \ref{t5} the estimator $\hat\theta$ is  stochastically increasing in $\theta$. Of course, $\hat\theta$
is also maximum likelihood estimator.
\end{exa}

%%%%%%%%%%%%%%%%%%%%%%%%%%%%%%%%%%%%%%%%%%%%%%%%%%%%%%%%%%%
%%%%%%%%%%%%%%%%%%%%%%%%%%%%%%%%%%%%%%%%%%%%%%%%%%%%%%%%%%%
%%%%%%%%%%%%%%%%%%%%%%%%%%%%%%%%%%%%%%%%%%%%%%%%%%%%%%%%%%%
%%%%%%%%%%%%%%%%%%%%%%%%%%%%%%%%%%%%%%%%%%%%%%%%%%%%%%%%%%%

\begin{exa}
Let  $\mathbf{X}=(X_1,\dots,X_n)$ be a sample 
from the gamma distribution with density 
$$
f(x;\lambda) = \frac{1}{\Gamma(\alpha) \lambda^\alpha} x^{\alpha -1} \exp (-x/\lambda), \ \ x>0, \ \lambda>0,
$$
where $\alpha>0$ is known.
This is clearly exponential family with $T(x)=x$ and $\eta(\lambda)=-1/\lambda$,  
so 
$$
E_{\lambda}(X)=\frac{\frac{\partial}{\partial \lambda} (-\alpha \log \lambda - \Gamma(\alpha))} { \frac{\partial}{\partial\lambda}(\frac{1}{\lambda})}
=\alpha \lambda.
$$
By   Theorem \ref{t3} the  
estimator $\hat\lambda= \bar X/\alpha$ 
is stochastically
increasing in $\lambda$.
Moreover, if $\alpha\geq 1$, then the density $f$ is 
logconcave and by  Theorem \ref{t3} the estimator
$\hat\lambda$ for $\alpha\geq 1$
 is also increasing with respect to  likelihood ratio order.

On the other hand, assume now that  $\lambda$ is fixed and $\alpha$ is unknown.
This is also one-parameter exponential family with $T(x)=\log x$ and $\eta(\alpha)=\alpha$.
Using (\ref{m_exp})  we get
$$
E_{\alpha}(\log(X))=\frac{\partial}{\partial \alpha} (\log \Gamma(\alpha)+\alpha\log\lambda)=
\Psi(\alpha)+\log \lambda=m(\alpha),
$$
where $\Psi(\alpha)=\frac{\partial}{\partial \alpha}\log\Gamma(\alpha)$ 
is the digamma function.
By Theorem \ref{t5} the estimator $\hat\alpha=m^{-1}(\bar T)$
stochastically
increasing in $\alpha$. This estimator obtained by the method
of maximum likelihood was considered
in Example 2 of Balakrishnan and Mi \cite{bal_01}.
\end{exa}

\begin{exa}
Let $\mathbf{X}=(X_1,\dots,X_n)$  be a sample from the
logistic distribution with density 
$$
f(x;\theta)=\frac{\theta \exp{(-x)}}{(1+\exp{(-x)})^{(\theta+1)}},
 \ \ x\in\mathcal{R}, \ \theta>0.
$$
This is the one-parameter exponential family with $T(x)=\log(1+\exp(-x))$ 
and $\eta(\theta)=\theta$.
From (\ref{m_exp}) we get
$
E_{\theta}(T(X))= 1/\theta,
$
hence from Theorem \ref{t3} the estimator $\hat\theta=1/\bar T$
is stochastically increasing in $\theta$.
\end{exa}

\begin{exa}
Let  $\mathbf{X}=(X_1,\dots,X_n)$ be a sample from the
uniform distribution on the interval $(0,\theta)$, $\theta>0$.
Them moment estimator of $\theta$ based on the first empirical moment
is of the form $\hat\theta=2 \bar X$ and by Theorem \ref{t3} this estimator is increasing in $\theta$
with respect to  likelihood ratio order.
Consider another moment estimator  based on generalized moment $\bar g =1/n\sum_{i=1}^n \log{X_i}$.
Easy calculations show, that $E_{\theta}(g(X))=\log{\theta}~-~1$, where $g(x)=\log x$,
 thus from Corollary \ref{cor2}  
the moment estimator $\hat\theta=\exp(1/n\sum_{i=1}^n\log{X_i}-1)$ is 
also increasing in $\theta$
with respect to  likelihood ratio order.
\end{exa}

%%%%%%%%%%%%%%%%%%%%%%%%%%%%%%%%%%%%%%%%%%%%%%%%%%%%%%%%%%%%%%%%%%%%%%%%%%%
%%%%%%%%%%%%%%%%%%%%%%%%%%%%%%%%%%%%%%%%%%%%%%%%%%%%%%
%%%%%%%%%%%%%%%%%%%%%%%%%%%%%%%%%%%%%%%%%%%%%%%%%%%%%%
%%%%%%%%%%%%%%%%%%%%%%%%%%%%%%%%%%%%%%%%%%%%%%%%%%%%%%
%%%%%%%%%%%%%%%%%%%%%%%%%%%%%%%%%%%%%%%%%%%%%%%%%%%%%%%%%%%%%%%%%%%%%%%%%%%
Now we consider the case when  our family of distribution is the location family.
Then we can formulate the following theorem.

\begin{thm}\label{t_5}
Let $\mathcal{F}=\{ f(x-\theta),\theta\in\mathcal{R}\}$
 be the location family. Suppose that
 $\mu_1=E(X(0))=\int_{-\infty}^{\infty} xf(x)dx<\infty$.
Then the estimator
 $\hat\theta=\bar X-\mu_1$ is stochastically increasing in $\theta$. 
Moreover, if $f$ is logconcave then  $\hat\theta$ is also increasing in $\theta$ 
with respect to  likelihood ratio order.
\end{thm}
\begin{proof}
The estimator $\theta$ is stochastically increasing in $\theta$
since the family $\mathcal{F}$ is stochastically increasing in~$\theta$.
If we assume that $f$ is logconcave, i.e. $f(x-\theta)$ is $TP_2$,
then the family $\mathcal{F}$ has monotone likelihood ratio.
Since the convolution of logconcave functions is the logconcave function
we deduce that the estimator $\hat\theta$ is also increasing with
respect to  likelihood ratio order.
\end{proof}

Similar results we may obtain for a scale parameter family.
The ordering property  of the moment estimators of $\theta$
in this case is described by the following theorem.

\begin{thm}\label{t7} 
Let  $\mathcal{F}=\{1/\theta f(x/\theta),\theta\in R_+\}$, $x>0$, be the scale parameter family.
Suppose that exists $k$-th moment of the  random variable $X(1)$
with density $f(x)$ and  $E(X^k(1))=~\mu_k$.
 Then the moment estimator $\hat \theta=\sqrt[k]{ m_k/\mu_k}$ is stochastically
increasing in $\theta$. Moreover, if $Var(X(1))=\sigma^2<\infty$,
 where $S^2= m_2 - m_1^2$
is the sample variance, then the
another  moment estimator for $\theta$ 
 given by $\hat\theta=\sqrt{S^2}/\sigma$,
 is also stochastically increasing in $\theta$.
\end{thm}
\begin{proof}
The first part of theorem is obvious since the family $\mathcal{F}$
is stochastically increasing in $\theta$.
It is also easy to prove that if $\theta_1<\theta_2\in\Theta$, then
$X(\theta_1)\leq_{\textup{disp}}X(\theta_2)$
 and then the vector of spacings $\mathbf{U}=(U_{1},\dots,U_{n})$,
where $U_i=X_{i:n}-X_{i-1:n}$, $i=1,\dots,n$ (we put $X_{0:n}=0$),
is stochastically increasing in $\theta$ (see Oja \cite{oja}). 
Since
$$
S^2=\frac{1}{n^2}\sum_{1\leq i<j\leq n} (X_{(j)}-X_{(i)})^2 
=\frac{1}{n^2}\sum_{1\leq i<j\leq n} (U_{j}+U_{j-1}+\cdots+U_{i+1})^2 ,
$$
then the  sample variance is an increasing function of vector $\mathbf{U}$.
Thus the theorem follows from
the stochastic preserving  property of multivariate stochastic ordered vectors under monotone operations.
\end{proof}

\begin{exa}
Let  $\mathbf{X}=(X_1,\dots,X_n)$ be a sample from the
logistic distribution with density 
$$
f(x;\theta)=\frac{\exp(-(x-\theta))}{1+\exp(-(x-\theta))}, \ x\in\mathcal{R},
\ \theta\in\mathcal{R}.
$$
It is not difficult to see that density $f$ is logconcave and $E_{\theta}(X)=\theta$.
 By Theorem~\ref{t_5} the estimator  $\hat \theta =\bar X$ is increasing in $\theta$ with respect to 
 likelihood ratio order.

\end{exa}

\begin{exa}
Let  $\mathbf{X}=(X_1,\dots,X_n)$ be a sample from the
Weibull distribution with density 
$$
f(x;\theta)=(1/\theta) x^{1/\theta-1}\exp{(-x^{1/\theta})}, \ x>0, \ \theta>0.
$$
It is obvious that the family of these
 distributions is not stochastically ordered in the  
parameter $\theta$.
Let $T_i=-\log X_i$, $i=1,\dots,n$.
After easy calculation we have that $T_1=_{\textup{st}}\theta W_1$,
where $W_1$ is the Gumbel distribution with cumulative distribution
function $e^{-e^{-x}}$, $x\in\mathcal{R}$. Thus the distribution of
 $T_1$ belongs to 
the scale parameter family.
 It is known that $E(W_1)=\gamma$, where $\gamma$ 
is Euler constant and $Var(W_1)=\pi^2/6$, hence
$\hat\theta=\sqrt{6S^2_T}/\pi$ and $\hat\theta=\bar T/\gamma$
are moment estimators for~$\theta$, where $S^2_T=1/n\sum_{i=1}^n(T_i-\bar T)^2$,
but we can not apply here
Theorem \ref{t7} since the support of $W_1$ is not $\mathcal{R}_+$.
So, let us consider $Z_{i}=|\log X_i|$, $i=1,\dots,n$.
Then we have $Z_1=_{\textup{st}} \theta|W_1|$ and 
$\mu=E(Z_1)=\gamma-2Ei(-1)=1.01598$
approximately, where $Ei(x)=-\int_{-x}^{\infty}e^{-t}/tdt$.
 Also, after calculations we have
 $\sigma^2=Var(Z_1)=\pi^2/6+4(\gamma-Ei(-1))Ei(-1)=0.945889$
approximately. By Theorem \ref{t7} the estimator $\hat\theta=\sqrt{S^2_Z}/\sigma$
is stochastically increasing in $\theta$. The same is true  for the
estimator $\hat\theta=\bar T/\mu$.
\end{exa}

%%%%%%%%%%%%%%%%%%%%%%%%%%%%%%%%%%%%%%%%%%%%%%%%%%%%%%%%%%%%%%%%%
%%%%%%%%%%%%%%%%%%%%%%%%%%%%%%%%%%%%%%%%%%%%%%%%%%%%%%%%%%%%%%%%%

\end{document}